\documentclass[a4paper,11pt,reqno]{amsart}
\usepackage{amsmath, amsthm, amssymb, mathrsfs, amsfonts}
\usepackage{array}
\usepackage[active]{srcltx}
\usepackage{graphicx}
\usepackage[margin=3cm]{geometry}
\usepackage{cite}
\usepackage{hyperref}
\usepackage{enumerate}
\usepackage{multicol}
\newtheorem{theor}{Theorem}[section]

\newtheorem{examp}{Example}[section]
\newtheorem{lem}{Lemma}[section]
\newtheorem{prop}{Proposition}[section]
\newtheorem{cor}{Corollary}[section]
\newtheorem{quest}{Question}
\newtheorem{rem}{Remark}

\def\car{\mathop{\rm char}\nolimits}

\usepackage{tikz,xcolor,hyperref}

\definecolor{lime}{HTML}{A6CE39}
\DeclareRobustCommand{\orcidicon}{%
	\begin{tikzpicture}
	\draw[lime, fill=lime] (0,0)
	circle [radius=0.16]
	node[white] {{\fontfamily{qag}\selectfont \tiny ID}};
	\draw[white, fill=white] (-0.0625,0.095)
	circle [radius=0.007];
	\end{tikzpicture}
	\hspace{-2mm}
}

\foreach \x in {A, ..., Z}{%
	\expandafter\xdef\csname orcid\x\endcsname{\noexpand\href{https://orcid.org/\csname orcidauthor\x\endcsname}{\noexpand\orcidicon}}
}

\title[Duo, Reversible and Symmetric GR's]%
      {On Duo, Reversible and Symmetric Group Rings}

\author[B. S Fl\'orez-Burbano]{Brayan S. Fl\'orez-Burbano\orcidC{}}
\address{Brayan S. Fl\'orez-Burbano, Escuela de Matem\'aticas, Universidad Industrial de Santander}
\email{brayan2218071@correo.uis.edu.co}

\author[A. Holgu\'in-Villa]{Alexander Holgu\'in-Villa\orcidA{}}
\address{Alexander Holgu\'in-Villa, Escuela de Matem\'aticas, Universidad Industrial de Santander}
\email{aholguin@uis.edu.co}

\author[J. H. Castillo]{John H. Castillo\orcidB{}}
\address{John H. Castillo, Departamento de Matem\'aticas y Estad\'istica, Universidad de Nari\~no}
\email{jhcastillo@udenar.edu.co}


\begin{document}
\maketitle
 \noindent
  \renewcommand{\refname}{References}

%%%%%%%%%%%%%%%%%%%%%%%%%%%%%%%%%%%%%%%%%%%%%%%%%%%%%%%%%%%%%%%%%%

\begin{abstract}
Let $RG$ denote the group ring of the torsion group $G$ over a commutative ring $R$ with identity. In this paper
we present proofs of some statements that appear without to be proved in the literature. We establish the valid
implications between the ring-theoretic conditions duo, reversible, SI property and symmetric in the setting of
group rings. We further show that if the group ring $RG$ possesses any of these properties, then $G$ is
a Hamiltonian group and the characteristic of $R$ is either $0$ or $2$. Moreover, we characterize the same properties
in group rings $RG$ in the following cases: ($1)$ $RG$ is a semi-simple group ring and ($2$) $R$ is a semi-simple
ring and $G$ any group.

\end{abstract}

%%%%%%%%%%%%%%%%%%%%%%%%%%%%%%%%%%%%%%%%%%%%%%%%%%%%%%%%%%%%%%%%%%%%%%%%%%
\section{Introduction}
%%%%%%%%%%%%%%%%%%%%%%%%%%%%%%%%%%%%%%%%%%%%%%%%%%%%%%%%%%%%%%%%%%%%%%%%%%

Throughout this manuscript, $R$ always will denote an associative ring, which does not necessarily contains identity unless otherwise noted. A ring $R$ is called left (resp. right) duo if every left (resp. right) ideal is an ideal, and $R$ is called duo if it is both
left and right duo. The ring $R$ is reversible if for all $a, b\in R$, $ab=0$ implies $ba=0$. Finally, $R$ is called
symmetric if any $a, b, c\in R$, $abc=0$ implies $acb=0$.

It is known that in every commutative ring, the set of nilpotent elements forms an ideal that coincides with the intersection
of all prime ideals. This result also holds for certain non-commutative rings, which are the so-called {\it $2$-primal rings}.
Marks in \cite{Marks:02} presented the following table that summarizes the valid implications between some ring-theoretic
properties that give rise to the $2$-primal rings, where the rings do not necessarily contain identity:

\begin{figure}[htb]
\begin{align*}
      \begin{matrix}
      \textup{ reduced }   & \Rightarrow & \textup{ symmetric }& & \textup{ all prime ideals are completely prime } \\
      \Downarrow           &             & \Downarrow        &            & \Downarrow\\
      \textup{ reversible} & \Rightarrow & \textup{SI}    & \Rightarrow & \textup{ $2$-primal }
      \end{matrix}
\end{align*}
\caption{Relations between six ring-theoretic conditions. }
\label{f1}
\end{figure}

Reversible rings, as a natural common generalization of commutative rings and integral domains, were studied in \cite{Cohn:99} by P. M. Cohn.
In particular, he observed that the Köthe conjecture holds for the class of these rings.
The relationship between symmetric and reversible rings were discussed in \cite{Marks:02}. Moreover, he proved that the group
algebra $\mathbf{F}_2Q_8$ of the quaternions $Q_8$ over the field $\mathbf{F}_2$ is reversible, but not symmetric. In \cite{GuKis:04}
M. Gutan and A. Kisielewicz characterized reversible group algebras $\mathbf{F}G$ of torsion groups $G$ over fields $\mathbf{F}$. In particular,
they described all finite reversible group algebras which are not symmetric, extending a result of G. Marks. Y. Li and M. M. Parmenter investigated
in \cite{LiPar:07} reversible group rings $RG$ of torsion groups $G$ over commutative rings $R$, which gave continuity to the results of Gutan
and Kisielewicz. For group algebras $\mathbf{F}G$ of torsion groups $G$ over a field $\mathbf{F}$, H. E. Bell and Y. Li in \cite{BeLi:07} showed
that the duo and reversible properties are equivalent. Finally, W. Gao and Y. Li in \cite{GaoLi:11} described when the group ring $RG$ of a
torsion group $G$ over an integral domain $R$ is duo.

In this paper, we present the proof of some results that are asserted (without proof) in the mentioned articles and we also establish a
modification of Figure \ref{f1}. Moreover, for group rings  $RG$ of a non-abelian torsion group $G$ over a commutative ring $ R$
with identity, we design diagrams similar to Figure \ref{f1} (see Figure \ref{f2} and Figure \ref{f3}), in which, in addition to determining the precise relationships among the duo, reversible and symmetric properties, we give necessary conditions, in terms
of the group $G$ and the ring $R$, when $RG$ has these properties. Finally, we close this article by presenting some results, Theorem \ref{t6},
Corollary \ref{c1}, Theorem \ref{t7}, Theorem \ref{t8} and Theorem \ref{t9}, that emerged after trying to answer Question \ref{q1}, raised by Gao and Li in \cite{GaoLi:11}.

%%%%%%%%%%%%%%%%%%%%%%%%%%%%%%%%%%%%%%%%%%%%%%%%%%%%%%%%%%%%%%%%%%%%%%%%%%
\section{Preliminaries and notations}
%%%%%%%%%%%%%%%%%%%%%%%%%%%%%%%%%%%%%%%%%%%%%%%%%%%%%%%%%%%%%%%%%%%%%%%%%%

Let $R$ be a ring. We use $\eta(R)$, $\eta_\ast(R)$ and $\eta^\ast(R)$ to denote the set of all nilpotent elements, the lower nil-radical of $R$,
i.e., the intersection of all prime ideals of $R$, and the upper nil-radical of $R$, i.e., the sum of all nil ideals of $R$, respectively. It is
well-known that $\eta_{\ast}(R)\subset \eta^\ast(R)\subset \eta(R)$. \\

A ring-theoretic property between ``commutative'' and ``symmetric'' is as follows: a ring $R$ is called reduced if it contain no non-zero
nilpotent elements. Another property which generalize commutativity that implies the $2$-primal condition is known as the SI condition, defined
by for all $a, b\in R$, $ab=0$ implies $aRb=\left\{0\right\}$.\\

The following examples, in addition to illustrating the ring-theoretic properties of interest, will be useful in the last section.

\begin{examp}\label{ex1}
Let $D$ be a division ring.
\begin{enumerate}
 \item Since the only right and left ideals of a division ring $D$ are the trivial two-sided ideals, then $D$ is right and left duo
       ring. Therefore, $D$ is a duo ring. Moreover, in a division ring $D$ there are no zero divisors, thus $D$ is reversible,
       symmetric, and has the SI property.

 \item Let $R=M_n(D)$ the full ring of $n \times n$ matrices over $D$. If $n \geq 2$, we have
      \begin{align*}
       I=\left \{
      \begin{bmatrix}
       x_1    &    x_2    & \ldots & x_n \\
       0    &    0    & \ldots & 0 \\
       \vdots & \vdots  & \ddots & \vdots  \\
       0    &    0    & \cdots & 0
       \end{bmatrix}: \hspace{0.1cm} x_i \in  D \right \},
 \end{align*}
is a right ideal of $R$ but it is not a left ideal of $R$ and thus,  the ring $R$ is not a right duo. Similarly, we can prove that
$R$ is not a left duo either. Moreover, if $n \geq 2$ we also can show that $R$ does not have neither reversible, symmetric
nor SI properties. Therefore, $R$ is duo, reversible, symmetric, and has the SI property if and only if $n=1$.
\end{enumerate}
\end{examp}

In what follows, we present the proof, for the convenience of the readers, of some results that appear in the Figure \ref{f1}.

\begin{lem} \label{l2}
If $R$ is a reduced ring, then $R$ is reversible and symmetric.
\begin{proof}
We will prove that reduced implies symmetric. The arguments to establish that reduced implies reversible are similar.

Let $a, b, c\in R$ such that $abc=0$, then $(cab)^{2}=c(abc)ab=0$. Since $R$ is reduced, we have $cab=0$ and thus, $(bca)^{2}=b(cab)ca=0$
and again by hypothesis $bca=0$. Similarly, $cacb=0$ and $acbac=0$ since $(cacb)^{2}=cac(bca)cb=0$ and $(acbac)^{2}=acba(cacb)ac=0$.
Therefore, $(bac)^{2}=b(acbac)=0$ and thus, $bac=0$. Finally, $(acb)^{2}=ac(bac)b=0$ and by hypothesis $acb=0$. This finishes the proof.

The proof of the other statement is similar.
\end{proof}
\end{lem}

\begin{lem} \label{l3}	
If $R$ is either reversible or symmetric, then $R$ has the SI property.
\begin{proof}
Suppose that $R$ is reversible and let $a, b \in R$ such that $ab=0$. Then, by hypothesis  $ba=0$ and for all $r \in R$ we have
$$
0=0r=(ba)r=b(ar)=(ar)b.
$$
Therefore, $aRb=\lbrace 0 \rbrace$ and $R$ has the SI
property.

If $R$ is symmetric, the proof that it has the SI property follows in a similar way.
\end{proof}
\end{lem}

Note that in a ring $R$ which has the SI property, the ideal $RaR$ is nilpotent whenever $a  \in  \eta(R)$, and hence any nilpotent element $a$
is contained in $\eta_{\ast}(R)$. Thus, $\eta_{\ast}(R)=\eta(R)$ and then $R$ is a $2$-primal ring. Moreover, by {{\cite[Proposition 1.11]{Shin:73}}}
a necessary and sufficient condition for a ring $R$ to be $2$-primal is that every minimal prime ideal $\mathfrak{p}$ of  $R$ be completely prime
(i.e. $R/\mathfrak{p}$ be a domain). The above together with Lemma \ref{l2} and Lemma \ref{l3} justify the implications presented in Figure
\ref{f1}. We include the following statement in this diagram.\\

Recall that, the subsets $Ann_l(a)=\left\{r\in R: ra=0\right\}$ y $Ann_r(a)=\left\{r\in R: ar=0\right\}$, denote the left  and right
annihilator of $a$ in $R$, respectively.

\begin{theor}\label{theor1}
If the ring $R$ is right duo or left duo, then $R$ has the SI property. In particular, if $R$ is a duo ring, then $R$ has the SI property.
\end{theor}

\begin{proof}
Suppose that $R$ is a right duo ring and consider $a, b \in R$ such that $ab=0$. Then, $b \in Ann_r(a)$ and since $Ann_r(a)$ is a right ideal
of $R$, it is  also a left ideal of $R$. Let $r \in R$. Then, $rb \in Ann_r(a)$  and thus,  $arb=0$. Therefore, $aRb = \lbrace 0 \rbrace$ and $R$ has
the SI property. When $R$ is a left duo ring, the proof is similar. By definition of a duo ring, the last statement is now clear.
\end{proof}

Moreover, since in a one-sided duo ring $R$ every prime ideal is completely prime, we present the following modification of
Figure \ref{f1}.

\begin{figure}[htb]
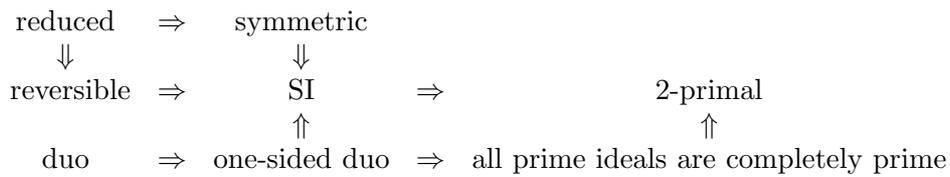

\begin{align*}
      \begin{matrix}
      \textup{ reduced }   & \Rightarrow & \textup{ symmetric }& &  \\
      \Downarrow           &             &    \Downarrow     &            & \\
      \textup{ reversible} & \Rightarrow & \textup{SI}    & \Rightarrow & \textup{ $2$-primal } \\
      &  & \Uparrow   &  & \Uparrow \\
      \textup{duo} &  \Rightarrow & \textup{one-sided duo}    & \Rightarrow & \textup{all prime ideals are completely prime}
      \end{matrix}
\end{align*}
\caption{Modified Marks' Diagram.}
\label{f4}
\end{figure}

We remark that Figure \ref{f4} allows us to prove that a duo ring is $2$-primal without setting that all its prime ideals
are completely prime.\\

Let $RG$ denote the group ring of a group $G$ over a commutative ring $R$ with identity. Using the so-called \emph{classical involution},
which is induced from the map $g\mapsto g^\ast=g^{-1}$, for all $g\in G$, given by $\Big(\sum_{g\in G} \alpha_gg\Big)^{\ast}=\sum_{g\in G} \alpha_gg^{-1}$,
Bell and Li \cite{BeLi:07} asserted, without proving, that the properties left and right duo are equivalent on $RG$. However,
for completeness’ sake, we now include easy proof of this fact.

\begin{prop}\label{prop1}
Let $R$ be a commutative ring with identity and $G$ be a group. The group ring $RG$ is left duo if and only if it is right duo.
\end{prop}

\begin{proof}
Suppose that the group ring $RG$ is left duo and let $I$ be a right ideal of $RG$. Then, the set
$J= \left \{ \alpha^{*}:\hspace{0.1cm} \alpha \in I \right \} \subseteq RG$ is a left ideal of $RG$. In fact,
$ J \neq \emptyset$ since $0_{RG}^{*} =0_{RG}\in J$ and for $\gamma \in RG$ and $x,y \in J$,  we have  $x=\alpha^{*}$ and
$y=\beta^{*}$ for some $\alpha, \beta \in I$. Therefore,
$$
x-y = \alpha^{*}-\beta^{*}=\left(\alpha-\beta \right)^{*} \in J, \quad \text{ and also } \quad \gamma x = \gamma \alpha^{*}=(\alpha \gamma^{*})^{*} \in J,
$$
thus $J$ is a left ideal of $RG$. Moreover,
\begin{align*}
J^{*}=\left \{w^{*}:\hspace{0.1cm}w \in J  \right \}=\left \{(\alpha^{*})^{*}:\hspace{0.1cm}\alpha \in I  \right \} = I.
\end{align*}
Finally, let $\gamma\in RG$ and $\rho\in I$. By hypothesis $RG$ is left duo, then $J$ is a two-sided ideal and since $\rho^\ast\in J$,
$\gamma^\ast \in RG$, we have $\rho^\ast\gamma^\ast\in J$, hence
\begin{align*}
\gamma\rho = \left((\gamma\rho)^{\ast}\right)^{\ast} = \left(\rho^{*}\gamma^{*}\right)^{*} \in J^{*}=I.
\end{align*}
Therefore, $RG$ is a right duo ring. The proof of the other implication is similar.
\end{proof}

\begin{rem}
The ``Only if" part in the previous proof can be schematized as follows.
\begin{figure}[htb]
\begin{align*}
      \begin{matrix}
      \textup{$^*: RG$}  & \longrightarrow  & \textup{$RG$} & \longrightarrow & \textup{$RG$}  &\\
      \textup{$I$ (right)}  & \longmapsto  & \textup{$I^*$} \textup{(left)} &  & \textup{$I^{**}=I$}  &\\
               &             & \Big\Downarrow  \textup{(Hypothesis)}      &            & \Big\Uparrow \\
      &  & \textup{$I^*$ (right)}    &  \longmapsto  & \textup{$(I^*)^*$ (left)}
      \end{matrix}
\end{align*}
\caption{Proof's Summary of Proposition  \ref{prop1}.}
\label{fd}
\end{figure}
\end{rem}

Marks \cite{Marks:02} showed that if $R$ is a commutative ring with identity and $G$ is a finite group, then the group ring $RG$ is reversible if
and only if it has the SI property. Bell and Li in \cite{BeLi:07} asserted, without proving, that this result is also true for any group $G$. In the sequel,
we present a proof of this statement.

\begin{prop} \label{resi}
Let $R$ be a commutative ring with identity and let $G$ be any group. Then, the group ring $RG$ is reversible if and only if $RG$ has the SI property.
\end{prop}

\begin{proof}
By Lemma \ref{l3}, in general reversible implies the SI condition.

Conversely, suppose that $RG$ has the SI property and let $\alpha, \beta \in RG$ such that $\alpha\beta=0$. Defining the trace map $tr: RG \rightarrow R$
in the usual way, $\sum_{g \in G}a_gg\mapsto a_1$, we  have $tr(xy) = tr(yx)$ for all $x, y \in RG$. Since $RG$ satisfies the SI condition,
$\alpha\gamma\beta=0$ for all $\gamma\in RG$. Therefore, $0=tr(\alpha[\gamma\beta]) = tr([\gamma\beta]\alpha)$ and thus, every element of the left ideal
$(RG)\beta\alpha$ has trace zero. Now, if $\beta\alpha$ has the following representation
\begin{align*}
\beta\alpha=\sum_{i=1}^{n}a_ig_i, \textup{  ~$a_i \in R$, ~$g_i \in G$ \quad and \quad $g_i \neq g_j$ \quad for \quad $i \neq j$,}
\end{align*}
then when considering the elements $\gamma_i = g_i^{-1} \in RG$, $1 \leq i \leq n$, it follows from $tr(\gamma_1 \beta\alpha)=0$ that $a_1=0$. Similarly,
from $tr(\gamma_2 \beta\alpha)=0$, it is obtained that $a_2=0$ and continuing in this way, $a_i=0$, $1 \leq i \leq n$, i.e., $\beta\alpha =0$ and thus,
$RG$ is reversible.

\end{proof}

%%%%%%%%%%%%%%%%%%%%%%%%%%%%%%%%%%%%%%%%%%%%%%%%%%%%%%%%%%%%%%%%%%%%%%%%%%
\section{Duo, reversible and symmetric group rings}
%%%%%%%%%%%%%%%%%%%%%%%%%%%%%%%%%%%%%%%%%%%%%%%%%%%%%%%%%%%%%%%%%%%%%%%%%%

Let $R$ be a commutative ring and $G$ be a group. If $G$ is an abelian group, then clearly $RG$ is commutative and thus,
$RG$ is duo, reversible, symmetric, and has the SI property. So we are interested in the case when $G$ is non-abelian. For a
non-abelian torsion group $G$ and $\mathbf{F}$ a field, Gutan and Kisielewicz {{\cite[Theorem 2.1]{GuKis:04}}} showed that if the group algebra $\mathbf{F}G$
is reversible, then $G$ is a Hamiltonian group. Y. Li and M. M. Parmenter in {{\cite[Introduction]{LiPar:07}}} asserted, without showing,
that this result is also true when $\mathbf{F}=R$ is a commutative ring with identity.

\begin{theor} \label{trihg}
Let $G$ be a non-abelian torsion group and $R$ be a commutative ring with identity. If the group ring $RG$ is reversible, then $G$ is Hamiltonian.
\end{theor}

\begin{proof}
We will show that for all $x \in G$, $\langle x \rangle$ is a normal subgroup of $G$, and thus, $G$ is a Hamiltonian group. When $x$ is the neutral
element, the result is clear. Let $x \in G \backslash \lbrace 1 \rbrace$. Since $G$ is a torsion group, $x$ has finite order. Let $n$ denote the order
of $x$ and let $g \in G$, then
\begin{align*}
   0=g(1-x^{n})=g(1-x)(1+x+x^{2}+\cdots+x^{n-1}),
\end{align*}
and since $RG$ is reversible,
\begin{align*}
0=(1+x+x^{2}+\cdots+x^{n-1})g(1-x)=(1+x+x^{2}+\cdots+x^{n-1})(g-gx)
\end{align*}
Multiplying by the inverse of $g$, we have
\begin{align}\label{eq:star}
   0&=(1+x+x^{2}+\cdots+x^{n-1})(1-gxg^{-1}).
\end{align}

Then, since the order of $g$ is $n$, the set $\lbrace 1, x,x^{2},\ldots,x^{n-1}\rbrace$ contains $n$ pairwise distinct elements of $G$ and from the expression \eqref{eq:star}, which is equivalent to $1+x +x^{2}+\cdots+x^{n-1}=gxg^{-1}+xgxg^{-1}+x^{2}gxg^{-1}+\cdots+x^{n-1}gxg^{-1}$, thus for $i$
there exists $j$ such that $x^{i}=x^{j}gxg^{-1}$, whence $gxg^{-1} \in \langle x \rangle$, then it follows that
$\langle x \rangle \trianglelefteq G$ for all $x \in G $.
\end{proof}

It is clear that if a ring $R$ with identity is symmetric, then it is also reversible. In fact, if $a, b\in R$ such that $ab=0$ and $R$ is symmetric,
then $1ab=1ba=ba=0$, and thus, $R$ is reversible. From the above and together with Theorem \ref{theor1}, Proposition \ref{prop1} Proposition \ref{resi}, and Theorem
\ref{trihg}, we obtain the following diagram on group rings $RG$ which summarizes the valid implications between  duo, reversible, symmetric, and SI
properties and allows us to conclude that a necessary condition for $RG$ to have these properties is $G$ to be Hamiltonian.

\begin{figure}[htb]
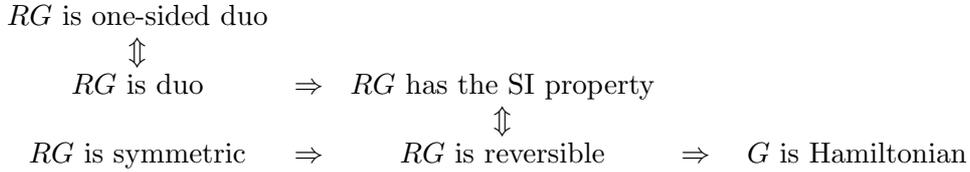

\begin{align*}
      \begin{matrix}
       \textup{$RG$ is one-sided duo}&  &     & &\\
       \Updownarrow & &   & &\\
       \textup{$RG$ is duo}& \Rightarrow & \textup{$RG$ has the SI property}    & &\\
       & & \Updownarrow    & &\\
      \textup{$RG$ is symmetric}  & \Rightarrow & \textup{$RG$ is reversible}& \Rightarrow & \textup{ $G$ is Hamiltonian}
      \end{matrix}
\end{align*}
\caption{Necessary condition over $RG$ and structure of $G$.}
\label{f2}
\end{figure}

It is well-known that every Hamiltonian group $G$ is of the form $G=Q_8 \times  E \times A$, where $Q_8=\langle x, y: x^4=1, x^2=y^2, y^{-1}xy=x^{-1}\rangle$
is the quaternion group of order $8$, $E$ is an elementary abelian $2$-group, and $A$ is an abelian group in which every element has odd order, see
\cite[Theorem 1.8.5]{PS:02}. Thus, by the previous theorem, if $G$ is a non-abelian torsion group and $R$ is a commutative ring with identity,
$RG=R(Q_8 \times E \times A)$ and since $R(Q_8 \times E \times A) \cong [R(E\times A)]Q_8$, the study of reversibility of $RG$ is essentially reduced to
the case when $G = Q_8$.

Li and Parmenter in \cite{LiPar:07} characterized when the group ring $\mathbb{Z}_nQ_8$ is a reversible ring. They proved the following theorem.

\begin{lem} [{{\cite[Theorem 2.5]{LiPar:07}}}]
Let $\mathbb{Z}_n$ be the ring of integers modulo $n$. The group ring $\mathbb{Z}_nQ_8$ is reversible if and only if $n=2$.
\end{lem}

Note that if $RQ_8$ is reversible and $\car(R) = n > 0$, then the subring $\mathbb{Z}_nQ_8$ of $RQ_8$ is also reversible, and by the last
lemma, $n = 2$. Then, if $RQ_8$ is reversible, the only possibilities to the characteristic of $R$ are either $\car(R)=0$ or $\car(R)=2$. Assume that $RG$ is reversible  with $G$ a non-abelian torsion group, by Theorem \ref{trihg}, $G$ is Hamiltonian. Since $RG\cong [R(E \times A)]Q_8$ is reversible and  $\car(R)=\car(R(E \times A))$, then either $\car(R)=0$ or $\car(R)=2$. Therefore, for non-abelian torsion groups, we have the following result.

\begin{theor}
Let $R$ be a commutative ring with identity and $G$ be a non-abelian torsion group. If $RG$ is reversible, then either $\car(R)=0$ or $\car(R)=2$.
\end{theor}

Using the previous theorem and in a similar way to the construction of Figure \ref{f2}, we present the following diagram for group
rings $RG$ of non-abelian torsion groups over any commutative rings with identity, where now some necessary conditions for $RG$ to have
these properties are that the characteristic of $R$ is either $0$ or $2$.

\begin{figure}[htb]
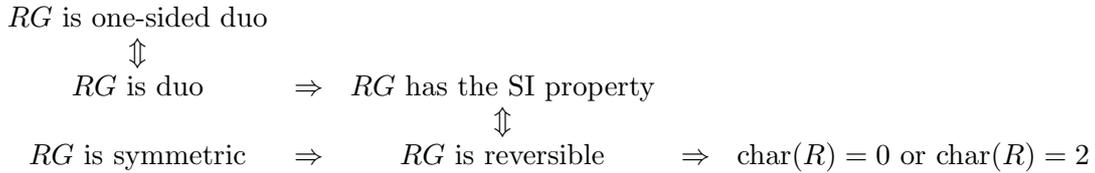

\begin{align*}
      \begin{matrix}
       \textup{$RG$ is one-sided duo}&  &     & &\\
       \Updownarrow & &   & &\\
       \textup{$RG$ is duo}& \Rightarrow & \textup{$RG$ has the SI property}    & &  \\
       & & \Updownarrow    &  &\\
      \textup{$RG$ is symmetric}  & \Rightarrow & \textup{$RG$ is reversible} & \Rightarrow &  \textup{$\car(R)=0$ or $\car(R)=2$}
      \end{matrix}
\end{align*}
\caption{Necessary condition over $RG$ and $\car(R)$.}
\label{f3}
\end{figure}

We close this section by presenting a summary of the different researches that have been developed around the ring-theoretic
properties of interest in this paper, in the setting of group rings.

Let $\mathbf{F}$ be a field and let $G$ be a non-abelian torsion group. According to Figure \ref{f3}, some necessary
conditions for $\mathbf{F}G$ to be a reversible ring are that $char(\mathbf{F})=0$ or $char(\mathbf{F})=2$. In
{{\cite[Theorem 5.1]{GuKis:04}}} Gutan and Kisielewicz gave a characterization of the reversibility of $\mathbf{F}G$ in terms
of these two options. Moreover, in {{\cite[Theorem 5.2]{GuKis:04}}} they described all finite reversible group rings. Li and
Parmenter investigated in \cite{LiPar:07} reversible group rings $RG$ of a torsion group $G$ over commutative rings, extending
the results of Gutan and Kisielewicz.

From Figure \ref{f2} it is known that if $G$ is a non-abelian
torsion group, duo implies reversible. However, the converse is not true, Bell and  Li in {{\cite[Example 1.2]{BeLi:07}}} showed
that the group ring $\mathbb{Z} Q_8$ is duo but is not reversible and using the results of Gutan and Kisielewicz, they established in
{{\cite[Theorem 4.2, Corollary 4.3]{BeLi:07}}} the equivalence between these two properties when $R=\mathbf{F}$ is a field and $G$
is a torsion group.

In order to extend the study of the duo property in group rings, Gao and Li in \cite{GaoLi:11} showed that if the group ring $RQ_8$
over an integral domain $R$ is duo, then $R$ is a field for the following cases: $(1)$ $\car(R) \neq 0$, $(2)$ $\car(R)=0$, and
$S \subseteq R \subseteq K_S$, where $S$ is a ring of algebraic integers and $K_S$ is its quotient field, and using this, they
described when the group ring $RG$ with $R$ an integral domain and $G$ a torsion group is duo, see {{\cite[Theorem 2.8, Theorem 2.9]{GaoLi:11}}}.

%%%%%%%%%%%%%%%%%%%%%%%%%%%%%%%%%%%%%%%%%%%%%%%%%%%%%%%%%%%%%%%%%%%%%%%%%%%
\section{Further Results}
%%%%%%%%%%%%%%%%%%%%%%%%%%%%%%%%%%%%%%%%%%%%%%%%%%%%%%%%%%%%%%%%%%%%%%%%%%%

Gao and Li in \cite{GaoLi:11} showed that if $R$ is an integral domain of $\car(R) = 0$, a necessary condition for $RQ_8$
to be duo is that $1+x^{2}+y^{2} \in \mathcal{U}(R)$ (the set of units of $R$) for all $x, y \in  R$. Moreover, they were
not aware of any example of an integral domain $R$ of $\car(R) = 0$ satisfying this necessary condition for which $RQ_8$
is not duo. Finally, they close their paper by proposing the following question:

\begin{quest} [{{\cite[Question 2.10]{GaoLi:11}}}]\label{q1}
Assume that $R$ is an integral domain of $\car(R) = 0$ such that $1+x^{2}+y^{2} \in \mathcal{U}(R)$ for all $x, y \in  R$. Is $RQ_8$ duo?
\end{quest}

We also do not know the answer to the previous question. However, modifying some hypotheses and in order to give continuity to the study of the relationships between the duo, reversible, symmetric and SI properties, in the setting of group rings, we present some results in this direction using the following lemma, whose proof is not hard.

\begin{lem} \label{l1}
Let $R_1, \ldots, R_k$ be rings such that $R=\bigoplus_{i=1}^{k}R_i$. Then, the ring $R$ is duo (resp. reversible, symmetric, has the SI property) if
and only if $R_i$ is duo (resp. reversible, symmetric, has the SI property), for $1 \leq i \leq k$.
\end{lem}

\begin{theor}\label{t6}
Let $R$ be a semi-simple ring and $G$ be a finite group, such that $\vert G \vert \in \mathcal{U}(R)$. The following statements are equivalent.
\begin{multicols}{2}
\begin{enumerate}
      \item $RG$ is duo.
      \item $RG$ is symmetric.
      \item $RG$ is reversible.
      \item $RG$ has the SI property.
      \item $RG \cong \bigoplus_{i=1}^{k}D_i$,  where $D_i$ are division rings.
\end{enumerate}
\end{multicols}
\end{theor}

\begin{proof}
We will show that $(1)\Leftrightarrow (5)$. The arguments to prove that $(2)\Leftrightarrow (5)$, $(3)\Leftrightarrow (5)$
and $(4)\Leftrightarrow (5)$ are similar.

Since $R$ is semi-simple, $G$ is finite and $\vert G \vert \in \mathcal{U}(R)$, by Maschke's theorem  \cite[Theorem 3.4.7]{PS:02},
the group ring $RG$ is semi-simple and thus, using the Wedderburn-Artin theorem {{\cite[Theorem 2.6.18]{PS:02}}}, we have that
\begin{align*}
      RG \cong \bigoplus_{i=1}^{k}M_{n_i}(D_i), \textup{ where $D_i$, $1\leq i\leq k$, is a division ring.}
\end{align*}

Suppose that $RG$ is a duo ring. Then, each $M_{n_i}(D_i)$ must also be duo, by Example \ref{ex1}, $n_i=1$, $1\leq i\leq k$, and thus $RG$ has the required decomposition.

Since the division rings are duo, the converse follows directly from Lemma \ref{l1}.
\end{proof}

Regarding Question \ref{q1}, by changing the hypotheses of the ring being an integral domain with characteristic zero for being
a semi-simple ring, from the previous theorem, we obtain the following corollary.

\begin{cor} \label{c1}
Let $R$ be a semi-simple ring, such that $1 + x^{2} + y^{2} \in \mathcal{U}(R)$ for all $x, y \in R$. Then, $RQ_8$ is a duo ring if and only
if $RQ_8$ is a finite direct sum of division rings, i.e.,
\begin{align*}
      RQ_8 \cong \bigoplus_{i=1}^{k}D_i, \textup{ where $D_i$, $1\leq i\leq k$, is a division ring.}
\end{align*}
\end{cor}

\begin{proof}
Note that $2=1+1^{2}+0^{2}$, then $2 \in \mathcal{U}(R)$. Since $4(2^{-1}2^{-1})=1$, it follows that $4 \in \mathcal{U}(R)$  and similarly,
$8 \in \mathcal{U}(R)$. Moreover, by hypothesis $R$ is semi-simple, then the result follows from Theorem \ref{t6}.
\end{proof}

\begin{examp}
Let $R= \left\{ \begin{pmatrix} r & s \\ 0 & r \end{pmatrix}: r, s \in \mathbb{Q} \right\}$. Given $x, y\in R$, it is not hard to check that
$\det\big(1+x^2+y^2\big)\neq 0$ and thus, $1 + x^{2} + y^{2} \in \mathcal{U}(R)$ for all $x, y \in R$. Following the arguments of the proof in
the last corollary, $|Q_8|\in \mathcal{U}(R)$. In the other hand, Y. Li and M. M. Parmenter in {{\cite[Example 4.3]{LiPar:07}}} showed that $RQ_8$ is
reversible. Furthermore, $R$ is a semi-simple ring, then by Theorem \ref{t6},
\begin{align*}
RQ_8 \cong \bigoplus_{i=1}^{k}D_i, \textup{ where $D_i$, $1\leq i\leq k$, is a division ring.}
\end{align*}

Therefore, $RQ_8$ is also duo, symmetric, and has SI property.
\end{examp}

It is possible to show that if the group ring $RG$ is duo, reversible, symmetric, or has the SI property, then $R$ must also respectively satisfy the property.
Next, we present a characterization of these properties for group rings $RG$ over semi-simple rings.

\begin{theor}\label{t7}
Let $R$ be a semi-simple ring, $G$ a group and $RG$ its group ring. Then, $RG$ is duo (resp. reversible, symmetric, has the SI property) if and only if
$R\cong \bigoplus_{i=1}^{k}D_i$ is a direct sum of division rings, and every group ring $D_iG$ is duo (resp. reversible, symmetric, has the SI
property), for $1 \leq i \leq k$.
\end{theor}

\begin{proof}
We will prove the duo case. The arguments in the other cases are similar.

Suppose that $RG$ is a duo ring. Then, $R$ must also be duo. Since $R$ is a semi-simple, it follows from Wedderburn-Artin theorem
{{\cite[Theorem 2.6.18]{PS:02}}}, that $R\cong \bigoplus_{i=1}^{k}M_{n_i}(D_i)$, where $D_i$ is a division ring, $1\leq i\leq k$.
By Lemma \ref{l1} every $M_{n_i}(D_i)$ is duo and thus, $n_i=1$. Then,
\begin{align*}
RG \cong \left( \bigoplus_{i=1}^{k}D_i \right) G\cong \bigoplus_{i=1}^{k}D_iG.
\end{align*}

Therefore, using the hypothesis and Lemma \ref{l1}, we have the result.

Conversely, if $R\cong \bigoplus_{i=1}^{k}D_i$ and every group ring $D_iG$ has duo property, we can write $RG \cong \bigoplus_{i=1}^{k}D_iG$
and again, it follows from Lemma \ref{l1} that $RG$ is also a duo ring.
\end{proof}

Finally, as a consequence of the Wedderburn-Artin theorem, we have: 1) If $R$ is a semi-simple ring which is commutative, then $R$ is isomorphic
to a finite direct sum of fields and 2) If $R$ is a finite-dimensional algebra over an algebraically closed field $\mathbf{F}$, which is semi-simple,
then  $R$ is isomorphic to a direct sum of the form $M_{n_1}(\mathbf{F})\oplus \cdots \oplus M_{n_k}(\mathbf{F})$
(see {{\cite[Exercises 3 and 8, p. 106]{PS:02}}}). Therefore, using similar arguments to the proof of the previous theorem, the following results
hold.

\begin{theor}\label{t8}
Let $R$ be a commutative semi-simple ring and $G$ a group. Then, the group ring $RG$ is duo (resp. reversible, symmetric, has the SI property) if and only if
$R \cong \bigoplus_{i=1}^{k}\mathbf{F}_i$ is a direct sum of fields, and each group ring $\mathbf{F}_iG$ is duo (resp. reversible, symmetric, has the SI property),
for $1 \leq i \leq k$. In particular, if $G$ is a non-abelian torsion group, then $G$ is Hamiltonian, and each $\mathbf{F}_i$ is of characteristic $0$ or $2$. 
\end{theor}

\begin{theor}\label{t9}
Let $R$ be a finite-dimensional algebra over an algebraically closed field $\mathbf{F}$, which is semi-simple. Then, $R$ has duo, reversible, symmetric or SI
property if and only if $R \cong \bigoplus_{i=1}^{k}\mathbf{F}$ is a direct sum of fields.
\end{theor}

\begin{examp}
Let $R= \mathbb{C} Q_8$, clearly $R$ is a finite-dimensional algebra over the algebraically closed field $\mathbb{C}$. Suppose that $R$ is a reversible ring,
then by Theorem \ref{t9}, $R \cong \bigoplus_{i=1}^{k} \mathbb{C}$, which is a contradiction since $R$ is not commutative. Therefore, $R$ is not reversible
and thus, from the Figure \ref{f2}, $R$ also does not have the duo, symmetric and SI properties.
\end{examp}

\begin{rem}
The previous example can also be deduced from Theorem 3.1 in  \cite{GuKis:04}.
\end{rem}

The following example shows a group ring over a commutative semi-simple ring which is reversible but not symmetric. This example was given by  Gutan and Kisielewicz in {{\cite[p. 291]{GuKis:04}}}, where they provided an example of a
reversible but not symmetric group ring over a ring of characteristic 0.

\begin{examp}
Let $R= \mathbf{F}_2 \oplus \mathbb{Q}$. It follows from Theorem 3.1 in  \cite{GuKis:04} that the rational group algebra $\mathbb{Q}Q_8$ is reversible and by {{\cite[Example 7]{Marks:02}}} we have that $\mathbf{F}_2 Q_8$ is also a reversible but not symmetric group ring. Therefore, by Theorem \ref{t8} $RQ_8$, is a  reversible but not symmetric group ring.
\end{examp}

%Finally, by Wedderburn-Artin Theorem {\cite[Proposition 6]{Lamb:09}}, if $R$ is a completely reducible ring, we have
%$ R \cong \bigoplus_{i=1}^{r}M_{n_i}(\mathbf{F}_i)$, where $\mathbf{F}_i$ is a field, $1 \leq i \leq r$, and thus, using similar arguments to the proof
%of the previous theorem, it is possible to prove the following result.

%\begin{theor}\label{t8}
%Let $R$ be a completely reducible ring, and $G$ a group. Then, the group ring $RG$ has duo (reversible, symmetric, SI) %property if only if
%$R \cong \bigoplus_{i=1}^{r}\mathbf{F}_i$ is a direct product of fields, and each group ring $\mathbf{F}_iG$ has duo (reversible, symmetric, SI) property,
%for $1 \leq i \leq r$. In particular, if $G$ is a non-abelian torsion group, then $G$ is Hamiltonian, and each $%\mathbf{F}_i$ is of characteristic $0$ or $2$.
%\end{theor}

%We remark that in the previous theorem, the equivalence for the reversible case is {\cite[Theorem 5.4.]{GuKis:04}}.

\section*{Acknowledgements}

The first and second authors are members of the Research Group ALCOM (\'Algebra y Combinatoria) at Escuela de Matem\'aticas - Universidad Industrial de Santander. The third author is
a member of the Research Group ALTENUA (\'Algebra, Teor\'ia de N\'umeros y Aplicaciones: ERM) and was partially supported by Vicerrectoria de Investigaciones e Interacci\'on Social at Universidad de Nari\~no.

\bibliographystyle{plain}
%\bibliography{bibliografia}  % associado ao arquivo: 'bibliografia.bib'

\end{document}